\documentclass[11pt]{article}
\usepackage{amssymb,amscd,graphicx}
\usepackage{fancyhdr}
 \usepackage[margin=1in]{geometry}
\usepackage{amsfonts}
\usepackage{amsmath, amsthm, mathrsfs}
\usepackage{color}
\usepackage{microtype}
\usepackage{hyperref}

\newtheorem{theorem}{Theorem}

\newtheorem{lemma}[theorem]{Lemma}

\theoremstyle{definition}

\let\oldmarginpar\marginpar
\renewcommand\marginpar[1]{\-\oldmarginpar[\raggedleft\footnotesize #1]%
{\raggedright\footnotesize #1}}

 \def\a{{\alpha}}
 \def\b{{\beta}}

 \def\l{{\lambda}}

 \newcommand{\nbd}{\ensuremath{\mbox{N}}} 

\title{Decomposition of a complete bipartite multigraph into arbitrary cycle sizes}
\author{John Asplund\\
\href{emailto:jasplund@daltonstate.edu}{jasplund@daltonstate.edu}\\
Dalton State College\\
Department of Technology and Mathematics\\
Dalton, GA 30720, USA\\
\\
Joe Chaffee\\
\href{joseph.chaffee@kp.org}{joseph.chaffee@kp.org}\\
Kaiser Permanente\\
3495 Piedmont Road\\ 
Building 9\\
Atlanta, GA 30305\\
\\
James M. Hammer\\
\href{emailto:jmhammer@cedarcrest.edu}{jmhammer@cedarcrest.edu}\\
Cedar Crest College\\
Allentown, PA 18104, USA\\
}

\begin{document}
\maketitle
%

\abstract{
In a graph $G$, let $\mu_G(xy)$ denote the number of edges between $x$ and $y$ in $G$. 
Let $\lambda K_{v,u}$ be the graph $(V\cup U,E)$ with $|V|=v$, $|U|=u$, and 
\[
\mu_G(xy)=\begin{cases} 
\lambda &\mbox{if $x\in U$ and $y\in V$ or if $x\in V$ and $y\in U$}\\
0 &\mbox{otherwise.} \\
\end{cases}
\] 
Let $M$ be a sequence of non-negative integers $m_1,m_2,\ldots,m_n$. An $(M)$-cycle decomposition of a graph $G$ is a partition of the edge set into cycles of lengths $m_1,m_2,\ldots,m_n$. In this paper, we establish necessary and sufficient conditions for the existence of an $(M)$-cycle decomposition of $\lambda K_{v,u}$.
}

\section{Introduction}
Let $G=(V(G),E(G))$ be a graph.  If $\{x,y\}\in E(G)$, we write $x\sim y$.  A graph $G$ is said to be even or odd if the degree of each vertex in $G$ is even or odd respectively. 
A $1$-factor in a graph $G$ is a subset $E\subseteq E(G)$ such that for each $x\in V(G)$, $x$ is incident with the edge $e$ for precisely one $e\in E$. 
A path of length $m$, an $m$-path, is a sequence $[x_0,x_1, \ldots, x_m]$ of $m+1$ distinct vertices such that $x_i \sim x_{i+1}$ are the only edges for all $i\in \mathbb{Z}_m$ (so the path has $m$ edges).  A cycle of length $m$, an $m$-cycle, is a graph with $V(G)=\mathbb{Z}_m$ and $E(G)=\{\{i,i+1\}\,:\, i\in \mathbb{Z}_m\}$ where $i+1$ is reduced mod $m$ and denoted as $(x_0,x_1,\ldots,x_{m-1})$.  A cycle decomposition of a graph $G$ is a partition of the edge set of $G$ such that each element of the partition induces a cycle. 
In a cycle decomposition, not all cycles must have the same length, but in the case where all the cycles do have the same length, say $m$, it is common to say that there exists an $m$-cycle decomposition of $G$.  A cycle packing of a graph $G$ is a cycle decomposition of a subgraph $H$ of $G$ ($H=G$ is allowed).  The leave of this packing is defined to be $E(G)\setminus E(H)$.  It is also sometimes helpful to think of the leave as the subgraph of $G$ with vertex set $V(G)$ and edge set equal to $L$. It should cause no confusion to use both definitions, and we will use both at different times in the paper depending on which is more natural in the situation. 
If $M=m_1,m_2,\ldots,m_t$ is a sequence of integers and there is a cycle packing whose partition contains $t$ elements and the $i^{th}$ element induces a cycle of length $m_i$, then, for notational convenience, we will call this cycle packing an $(M)$-cycle packing. 
Let $\mu_G(xy)$ be the number of edges that join $x$ and $y$ in $G$ and $\mu_G=\max(\{\mu_G(xy)\mid x,y \in V(G)\})$.
For notational convenience, for a sequence of integers $M$, define $\nu_k(M)$ to be the number of times $k$ appears in $M$.

Let $\lambda K_{v,u}$ be the graph $(V\cup U,E)$ with $|V|=v$, $|U|=u$, and 
\[
\mu_G(xy)=\begin{cases} 
\lambda &\mbox{if $x\in U$ and $y\in V$ or if $x\in V$ and $y\in U$}\\
0 &\mbox{otherwise} \\
\end{cases}.
\] 
The main result of this paper establishes necessary and sufficient conditions for an $(M)$-cycle decomposition of $\l K_{v,u}$. 

Decomposition problems have been studied heavily in the last several decades. Most of the focus has been on showing that there exists an $(M)$-cycle decomposition of $K_v$. Many steps were made along the way, but the final pieces were put together by Bryant et al. in {\normalfont\cite{BryantHorsleyPettersson}}, thus solving this existence problem. Though there were many results with uniform cycle lengths on a complete bipartite graph before, much of the progress for showing the existence of an $(M)$-cycle decomposition of $K_{v,u}$ when the cycles are of even length was made by Horsley in {\normalfont\cite{Horsley}}. 
All of these results were for simple graphs; here we look at multigraphs. The first result below leads to a necessary condition for the existence of an $(M)$-cycle decomposition of $\l K_{u,v}$.

\begin{theorem}\label{anyGraphNecessary}
{\normalfont\cite{bryant2015decompositions}} Suppose $G$ is a graph in which $\mu_G(xy)$ is even for each pair of vertices $x$ and $y$, $\mathcal{D}$ is a cycle decomposition of $G$, and $C\in \mathcal{D}$. Then $|\mathcal{D}|\leq |E(G)|/2-|E(C)|+2$.
\end{theorem}

\begin{theorem}\label{Nconditions}
Let $M=m_1,m_2,\ldots,m_t$ such that $m_1\leq m_2\leq\cdots \leq m_t$ and $m_i\equiv 0\pmod{2}$ for all $i\in\{1,2,\ldots,t\}$. If there exists an $(M)$-cycle decomposition of $\l K_{v,u}$ then all of the following hold:
\begin{itemize}
 \item[$(a)$] $m_t\leq 2\min(\{v,u\})$,
 \item[$(b)$] $\l v\equiv \l u\equiv 0\pmod{2}$,
 \item[$(c)$] $t\leq \frac{\l}{2}vu-m_t+2$ if $\l$ is even, and
 \item[$(d)$] $2\nu_2(M)\leq (\l-1)vu$ if $\l$ is odd.
\end{itemize}
\end{theorem}

\begin{proof}
Since $\l K_{v,u}$ is bipartite, any cycle must alternate between the two parts. Therefore Condition $(a)$ is necessary. In any cycle decomposition, each vertex must have even degree and thus Condition $(b)$ is necessary. Condition $(c)$ follows directly from Theorem~{\normalfont\ref{anyGraphNecessary}}. 
If $\lambda$ is odd then at least one edge between every pair of vertices cannot be used in a $2$-cycle. This shows that Condition $(d)$ is necessary. 
\end{proof}

Notice that a necessary condition for an $(M)$-cycle decomposition of $G$ is that $G$ is even. If $G$ is odd and $M=m_1,\ldots,m_t$ is a sequence of integers then we will say there exists an $(M)^*$-cycle decomposition of $G$ if $E(G)$ can be partitioned into a $1$-factor and $t$ cycles, the $i^{th}$ inducing a cycle of length $m_i$ for $1\leq i\leq t$.

Section~{\normalfont\ref{edgeSwitching}} will provide a cycle switching method that will be vital in Section~{\normalfont\ref{joiningCycles}}. Section~{\normalfont\ref{joiningCycles}} is dedicated to showing that if we have a cycle decomposition of $\l K_{v,u}$ that contains an $m_1$- and $m_2$-cycle, then we can join these two cycles together to form an $(m_1+m_2)$-cycle under certain conditions. This result will be the key lemma for the main theorem found in Section~{\normalfont\ref{mainTheorem}}.


\section{Edge Switching}\label{edgeSwitching}


The idea for cycle switching was first introduced for complete graphs by Bryant, Horsley, and Maenhaut in {\normalfont\cite{BryantHorsleyMaenhaut}}. Later it was used in other decompositions of complete graphs  {\normalfont\cite{BryantHorsley,BryantHorsley2,BryantHorsley3}} which culminated in the solution for the Alspach conjecture in {\normalfont\cite{BryantHorsleyPettersson}}. Further, this cycle switching method was adapted for other graphs and other decompositions in {\normalfont\cite{B,BryantHorsleyMaenhautSmith,Horsley,Horsley2}}. We will extend the cycle switching method to general multigraphs. Let the \textit{open neighborhood} of $x$ in $G$ be defined as $N_G(x)=\{y\in V\,:\, x\sim y\}$. Two distinct vertices $\alpha$ and $\beta$ are said to be \emph{twin} if $\nbd_G(\alpha)\setminus \{\beta\}=\nbd_G(\beta)\setminus \{\alpha\}$.

Given the permutation $\pi$ of a set $V$ and a graph $G=(V(G)=V,E(G))$, let $\pi(G)$ be the graph defined as having vertex set $V(G)$ and edge set $\{\pi(x)\pi(y)\,:\, xy\in E(G)\}$. 
The following theorem is an edge-switching result for $\l K_n$ which we will then extend to a general multigraph.

\begin{theorem}\label{lemma2.1}
{\normalfont\cite{BryantHorsleyMaenhautSmith}} Let $\alpha,\beta\in V(\l K_n)$ and let $\sigma$ denote the permutation $(\alpha\beta)$. If $H$ is a subgraph of $\l K_n$, then we define $\sigma(H)$ to be the subgraph $H'$ of $\l K_n$ with $V(H')=\sigma(V(H))$ and $\mu_{H'}(\sigma(a)\sigma(b))=\mu_H(ab)$ for any distinct $a,b\in V(H)$. Let $n$ and $\l$ be positive integers, let $\mathcal{P}$ be an $(M)$-cycle or $(M)^*$-cycle packing of $\l K_n$ when $\l K_n$ is even or odd respectively, let $L$ be the leave of $\mathcal{P}$, let $\alpha$ and $\beta$ be distinct vertices of $L$, and let $\sigma$ be as defined above.
Let $E$ be a subset of $E(L)$ such that, for each vertex $u\in V(L)\setminus \{\alpha,\beta\}$, $E$ contains precisely $\max(\{0,\mu_L(u\alpha)-\mu_L(u\beta)\})$ edges with endpoints $u$ and $\alpha$, and precisely $\max(\{0,\mu_L(u\beta)-\mu_L(u\alpha)\})$ edges with endpoints $u$ and $\beta$ (so $E$ may contain multiple edges with the same endpoints), and $E$ contains no other edges. Then there exists a partition $\tau$ of $E$ into pairs such that for each pair $\{x_1y_1,x_2y_2\}\in\tau$, there exists an $(M)$-cycle or $(M)^*$-cycle packing $\mathcal{P}'$ of $\l K_n$ with leave $L'=(L-\{x_1y_1,x_2y_2\})\cup\{\sigma(x_1)\sigma(y_1),\sigma(x_2)\sigma(y_2)\}$.

Furthermore, $P'$ also has the following properties. If $\mathcal{P}=\{C_1,C_2,\ldots,C_t\}$ ($\l(n-1)$ even) or $\mathcal{P}=\{F,C_1,C_2,\ldots,C_t\}$ ($\l(n-1)$ odd), where $C_1,C_2,\ldots,C_t$ are cycles and $F$ is a perfect matching, then $\mathcal{P}'=\{C_1',C_2',\ldots,C_t'\}$ ($\l(n-1)$ even) or $\mathcal{P}'=\{F',C_1',C_2',\ldots,C_t'\}$ ($\l(n-1)$ odd) where for $i=1,2,\ldots,t$ $C_i'$ is a cycle of the same length as $C_i$ and $F'$ is a perfect matching such that 
\begin{itemize}
\item $F'=F$ or $F'=\pi(F)$;
\item for $i=1,2,\ldots,t,$ if neither $\alpha$ nor $\beta$ is in $V(C_i)$, then $C_i'=C_i$;
\item for $i=1,2,\ldots,t,$ if exactly one of $\alpha$ and $\beta$ is in $V(C_i)$, then $C_i'=C_i$ or $C_i=\pi(C_i)$; and
\item for $i=1,2,\ldots,t,$ if both $\alpha$ and $\beta$ are in $V(C_i)$, then $C_i'=Q_i\cup Q_i'$ where $Q_i=P_i$ or $\pi(P_i)$, $Q_i'=P_i'$ or $\pi(P_i')$, and where $P_i$ and $P_i'$ are the two paths from $\alpha$ to $\beta$ in $C_i$.
\end{itemize}
\end{theorem}

If one reads the proof of Theorem~3 in \cite{BryantHorsleyMaenhautSmith}, they will notice that the key to the proof is the existence of twin vertices, which of course exist in $\lambda K_n$.
The next lemma is the analogue of Theorem~{\normalfont\ref{lemma2.1}} for some multigraph $G$. 

\begin{lemma}\label{cycleSwitching}
Let $M$ be a sequence of integers, let $\mathcal{P}$ be an $(M)$-cycle or $(M)^*$-cycle packing of a multigraph $G$ when $G$ is even or odd respectively, let $L$ be the leave of $\mathcal{P}$, and let $\alpha$ and $\beta$ be twin vertices of $L$.
Let $A$ be a subset of $E(L)$ such that, for each vertex $u\in V(L)\setminus \{\alpha,\beta\}$, $A$ contains precisely $\max(\{0,\mu_L(u\alpha)-\mu_L(u\beta)\})$ edges with endpoints $u$ and $\alpha$, and precisely $\max(\{0,\mu_L(u\beta)-\mu_L(u\alpha)\})$ edges with endpoints $u$ and $\beta$ (so $A$ may contain multiple edges with the same endpoints), and $A$ contains no other edges. Let $\pi$ be a permutation of $V$. Then there exists a partition $\tau$ of $A$ into pairs such that for each pair $\{x_1y_1,x_2y_2\}$ in the partition of $A$, there exists an $(M)$-cycle or $(M)^*$-cycle packing $\mathcal{P}'$ of $G$ with leave $L'=(L-\{x_1y_1,x_2y_2\})\cup \{\pi(x_1)\pi(y_1),\pi(x_2)\pi(y_2)\}$. 


Furthermore, $\mathcal{P}'$ also has the following properties: If $\mathcal{P}=\{C_1,C_2,\ldots,C_t\}$ ($G$ even) or $\mathcal{P}=\{F,C_1,C_2,\ldots,C_t\}$ ($G$ odd), where $C_1,C_2,\ldots,C_t$ are cycles and $F$ is a $1$-factor, then $\mathcal{P}'=\{C_1',C_2',\ldots,C_t'\}$ ($|V(G)|$ even) or $\mathcal{P}'=\{F',C_1',C_2',\ldots,C_t'\}$ ($|V(G)|$ odd) where for $i=1,2,\ldots,t$, $C_i'$ is a cycle of the same length as $C_i$ and $F'$ is a perfect matching such that
\begin{itemize}
\item $F'=F$ or $F'=\pi(F)$;
\item for $i=1,2,\ldots,t,$ if neither $\alpha$ nor $\beta$ is in $V(C_i)$, then $C_i'=C_i$;
\item for $i=1,2,\ldots,t,$ if exactly one of $\alpha$ and $\beta$ is in $V(C_i)$, then $C_i'=C_i$ or $C_i'=\pi(C_i)$; and
\item for $i=1,2,\ldots,t,$ if both $\alpha$ and $\beta$ are in $V(C_i)$, then $C_i'=Q_i\cup Q_i'$ where $Q_i=P_i$ or $\pi(P_i)$, $Q_i'=P_i'$ or $\pi(P_i')$, and where $P_i$ and $P_i'$ are the two paths from $\alpha$ to $\beta$ in $C_i$.
\end{itemize}
\end{lemma}

\begin{proof}
Let $G^\dagger$ be obtained from $G$ by adding an isolated vertex $\infty$ if both $|V(G)|$ is even and $G$ is even and let $G^\dagger=G$ otherwise. Note that we can consider $\mathcal{P}$ as an $(M)^*$-cycle packing of a complete graph with multiplicity $\mu_{G^\dagger}$ and $|V(G^\dagger)|$ vertices. The edge set of $G^\dagger$ is the same as $G$ in either case. Let the leave $H$ of $\mathcal{P}$ be the edge-disjoint union of $L$ and $\mu_{G^\dagger}-\mu_{G^{\dagger}}(xy)$ copies of each edge $xy\in E(G^\dagger)$. By applying Theorem~{\normalfont\ref{lemma2.1}}, we obtain a permutation $\pi$ such that for each $x\in (\nbd_L(\alpha)\cup \nbd_L(\beta))\setminus ((\nbd_L(\alpha)\cap \nbd_L(\beta))\cup\{\alpha,\beta\})$ there exists an $(M)$-cycle packing $\mathcal{P}'$ of $G^\dagger$ with leave $H'$ which differs from $H$ only in that each of the edges $\a x$, $\a \pi(x)$, $\b x$, $\b \pi(x)$ is an edge of $H'$ if and only if it is not an edge of $H$. Furthermore, by Theorem~{\normalfont\ref{lemma2.1}} edges are switched so that the following hold: 
\begin{itemize}
\item $F'=F$ or $F'=\pi(F)$;
\item for $i=1,2,\ldots,t,$ if neither $\alpha$ nor $\beta$ is in $V(C_i)$, then $C_i'=C_i$;
\item for $i=1,2,\ldots,t,$ if exactly one of $\alpha$ and $\beta$ is in $V(C_i)$, then $C_i'=C_i$ or $C_i=\pi(C_i)$; and
\item for $i=1,2,\ldots,t,$ if both $\alpha$ and $\beta$ are in $V(C_i)$, then $C_i'=Q_i\cup Q_i'$ where $Q_i=P_i$ or $\pi(P_i)$, $Q_i'=P_i'$ or $\pi(P_i')$, and where $P_i$ and $P_i'$ are the two paths from $\alpha$ to $\beta$ in $C_i$.
\end{itemize}
By the way we constructed $H$ and since $\nbd_G(\alpha)\setminus \{\beta\}=\nbd_G(\beta)\setminus\{\alpha\}$, it follows that
$$(\nbd_H(\alpha)\cup \nbd_H(\beta))\setminus ((\nbd_H(\alpha)\cap \nbd_H(\beta))\cup\{\alpha,\beta\}) = (\nbd_L(\alpha)\cup \nbd_L(\beta))\setminus ((\nbd_L(\alpha)\cap \nbd_L(\beta))\cup\{\alpha,\beta\}).$$
Let $L'$ be the leave on vertex set $V(G)$ and differ from $L$ only in that $\alpha x$, $\alpha \pi(x)$, $\beta x$, $\beta\pi(x)$ are edges in $L'$ if and only if they are not edges in $L$. Since $\alpha x, \alpha\pi(x), \beta x,\beta\pi(x)\in E(G)$, $H'=L'$ if $G=G^\dagger$ or $H'$ the edge-disjoint union of $L$ and $\mu_{G^\dagger}-\mu_{G^{\dagger}}(xy)$ copies of each edge $xy\in E(G^\dagger)$. Now looking at $L'$ as a leave of $G$, we have our result.
\end{proof}

Whenever Lemma~{\normalfont\ref{cycleSwitching}} is applied we have the ability to choose twin vertices $\alpha$ and $\beta$, and vertex $a$ such that $\mu_L(a\alpha)>\mu_L(a\beta)$ in which case $a\alpha\in E$. To simplify this concept, we say that an $(M)^*$-cycle packing $\mathcal{P}'$ was obtained from $\mathcal{P}$ when we perform an \textit{$(\alpha,\beta)$-switch} from an $(M)^*$-cycle packing $\mathcal{P}$ with \emph{origin} $a$ and \emph{terminus} $b$ (equivalently origin $b$ and terminus $a$). 
By applying Lemma~{\normalfont\ref{cycleSwitching}}, notice that if the leave $L$ is a simple graph for some packing $\mathcal{P}$ of a multigraph $G$, then the leave $L'$ of $\mathcal{P}'$ obtained by performing an $(\alpha,\beta)$-switch with origin $a$ will also be a simple graph. In fact, $L'$ differs from $L$ in Lemma~{\normalfont\ref{cycleSwitching}} only in that each of the edges $\alpha a$, $\alpha b$, $\beta a$, and $\beta b$ is an edge in $L'$ if and only if it is not an edge of $L$. So if $L_1$ is a leave for a packing $\mathcal{P}_1$ on a simple graph $G_1$, $L_2$ is the leave for a packing $\mathcal{P}_2$ on a multigraph $G_2$, $L_1=L_2$, and $\alpha$ and $\beta$ are twin, then the leaves $L_1'$ and $L_2'$ obtained by performing an $(\alpha,\beta)$-switch with origin $a$ and terminus $\pi(a)$ (where the vertices $\alpha$, $\beta$, and $a$ correspond to the same vertices in both $L_1$ and $L_2$) are equivalent if this switch has the same terminus in $L_1$ and $L_2$ (the terminus may not be the same vertex in $L_1$ as $L_2$ but we know that the choices for the terminus must be the same in $L_1$ and $L_2$). This observation allows us to use the lemmas proved in {\normalfont\cite{Horsley}} whose proofs only use results stemming directly from a simple graph version of the cycle switching method in Theorem~{\normalfont\ref{lemma2.1}} to prove results related to cycle decompositions of $\l K_{v,u}$ if the leave is a simple graph. That is, Theorems~{\normalfont\ref{lemma2.3H}}-{\normalfont\ref{lemma3.6H}} were all proven using Theorem 2.1 in \cite{BryantHorsley2} (but could have used Theorem~\ref{lemma2.1} since Theorem~\ref{lemma2.1} works on multigraphs and simple graphs while Theorem 2.1 in \cite{BryantHorsley2} only works on simple graphs and they do the same transformation to the leave), and so can be used to help prove multigraph analogues to Theorems~{\normalfont\ref{lemma2.3H}}-{\normalfont\ref{lemma3.6H}} whenever the leave is a simple graph. 

\section{Joining Cycles}\label{joiningCycles}

The aim of this section is to prove that we can modify a cycle decomposition to join two specific cycles together into a single cycle and all other cycles remain the same size. We begin with several theorems from Horsley {\normalfont\cite{Horsley}}, then prove several analogue theorems for the multigraph case before proving the main result of this section in Lemma~{\normalfont\ref{multiBipartitePacking}}. The proof of Lemma~\ref{multiBipartitePacking} is very similar to that found in Lemma~3.6 in {\normalfont\cite{Horsley}} and could be written almost verbatim into the proof of Lemma~\ref{multiBipartitePacking}. We will use this similarity to shorten or omit proofs when the proofs in {\normalfont\cite{Horsley}} are enough.
In {\normalfont\cite{Horsley}}, Lemmas~2.3, 2.5, 2.6, and 3.1--3.5 were required to prove Lemma~3.6 and so we will need the analogous versions of these lemmas for the multigraph case to prove Lemma~{\normalfont\ref{multiBipartitePacking}}. These lemmas from {\normalfont\cite{Horsley}} are provided below in order. 

We use the same definitions found in {\normalfont\cite{Horsley}} for the following terms. A \emph{chain} is a collection of cycles $A_1,A_2,\ldots,A_r$ such that
\begin{itemize}
\item $A_i$ is a cycle of length $a_i\geq 2$, and 
\item for $1\leq i<j\leq r$, $|V(A_i)\cap V(A_j)|=1$ if $j=i+1$ and $|V(A_i)\cap V(A_j)|=0$ otherwise.
\end{itemize}
The cycles $A_1$ and $A_r$ are called \emph{end cycles} and the vertex in $A_i\cap A_{i+1}$ is called the \emph{link vertex}. A chain containing $r$ cycles is called an \emph{$r$-chain}. A $2$-chain with cycles $C_1$ and $C_2$ will be denoted $C_1\cdot C_2$ or as a $(c_1,c_2)$-chain where $c_1$ and $c_2$ are the lengths of $C_1$ and $C_2$ respectively. 

A \emph{ring} is a collection of cycles $A_1,A_2,\ldots,A_r$ such that 
\begin{itemize}
\item $A_i$ is a cycle of length $a_i\geq 2$,
\item if $r\geq 3$ and $1\leq i<j\leq r$, then $|V(A_i)\cap V(A_j)|=1$ if $j=i+1$ or $(1,r)=(i,j)$, and $|V(A_i)\cap V(A_j)|=0$ otherwise, and
\item if $r=2$ then $|V(A_1)\cap V(A_2)|=2$.
\end{itemize}
Each cycle $A_1,A_2,\ldots,A_r$ is called a \emph{ring cycle}.

\begin{theorem}\label{lemma2.3H}
{\normalfont\cite{Horsley}} Let $a$ and $b$ be positive integers. Suppose that there exists an $(M)^*$-cycle packing $\mathcal{P}$ of $K_{a,b}$ with a leave of size $\ell$ whose only non-trivial component $H$ contains a path $P=[x_0,x_1,\ldots,x_t]$ of even length at least $4$ such that the edges in $E(H)\setminus E(P)$ form a path and such that $x_1x_t\not\in E(H)$. Let $S$ be the $(x_0,x_t)$-switch with origin $x_1$ (note that $x_0$ and $x_t$ are in the same part of $K_{a,b}$ since  $P$ is a path of even length) and let $\mathcal{P}'$ be the $(M)^*$-cycle packing of $K_{a,b}$ obtained from $\mathcal{P}$ by performing $S$. If $S$ does not have terminus $x_{t-1}$, then the leave of $\mathcal{P}'$ has a decomposition into a $t$-cycle and an $(\ell-t)$-cycle, and there are at least as many vertices of degree $4$ in the leave of $\mathcal{P}'$ as there are in the leave of $\mathcal{P}$. 
\end{theorem}

\begin{theorem}\label{lemma2.5H}
{\normalfont\cite{Horsley}} Let $a$ and $b$ be positive integers. Suppose there exists an $(M)^*$-cycle packing of $K_{a,b}$ with a leave whose only non-trivial component is a $(p,q)$-chain. If $m$ is an even integer such that $p\leq m$ and $p+q-m\geq 4$, then there exist
\begin{enumerate}
\item[$(i)$] an $(M)^*$-cycle packing of $K_{a,b}$ with a leave whose only non-trivial component either has a decomposition into an $m$-cycle and a $(p+q-m)$-cycle, or is an $(m-p+2,2p+q-m-2)$-chain. 
\item[$(ii)$] an $(M)^*$-cycle packing of $K_{a,b}$ with a leave whose only non-trivial component either has a decomposition into an $m$-cycle and a $(p+q-m)$-cycle, or is an $(m-p+4,2p+q-m-4)$-chain.
\end{enumerate}
\end{theorem}

\begin{theorem}\label{lemma2.6H}
{\normalfont\cite{Horsley}} Let $a$ and $b$ be positive integers. Suppose there exists an $(M)^*$-cycle packing of $K_{a,b}$ with a leave whose only non-trivial component is a $(p,q)$-chain. If $m_1$ and $m_2$ are even integers such that $m_1,m_2\geq 4$ and $m_1+m_2=p+q$, then there exists an $(M,m_1,m_2)^*$-cycle decomposition of $K_{a,b}$. 
\end{theorem}

\begin{theorem}\label{lemma3.1H}
{\normalfont\cite{Horsley}} Let $a$ and $b$ be positive integers. Suppose there exists an $(M)^*$-cycle packing of $K_{a,b}$ where $a\leq b$, with a leave of size $\ell$, where $\ell\leq 2a+2$ if $a<b$ and $\ell\leq 2a$ if $a=b$, with only one non-trivial component $H$. If $m_1$ and $m_2$ are even integers such that $m_1,m_2\geq 4$ and either
\begin{itemize}
\item $H$ is a chain which has a decomposition into an $m_1$-path and an $m_2$-path; or
\item $H$ is a ring which has a decomposition into an $m_1$-path and an $m_2$-path;
\end{itemize}
then there exists an $(M,m_1,m_2)^*$-cycle decomposition of $K_{a,b}$.
\end{theorem}

\begin{theorem}\label{lemma3.2H}
{\normalfont\cite{Horsley}} Let $a$ and $b$ be positive integers. Suppose that there exists an $(M)^*$-cycle packing of $K_{a,b}$ with a leave $L$ of size $\ell$ with $k$ non-trivial components such that exactly one vertex of $L$ has degree $4$ and every other vertex of $L$ has degree $2$ or degree $0$. If $R$ is one of the parts of $K_{a,b}$ and $m_1$ and $m_2$ are integers such that $m_1,m_2\geq k+1$ and $m_1+m_2=\ell$, then there exists an $(M)^*$-cycle packing of $K_{a,b}$ with a leave whose only non-trivial component is a chain which has a decomposition into an $m_1$-path and an $m_2$-path such that if $m_1,m_2\geq 3$ then at least one end vertex of the paths is in $R$.
\end{theorem}

\begin{theorem}\label{lemma3.3H}
{\normalfont\cite{Horsley}} Let $a$ and $b$ be positive integers. Suppose that there exists an $(M)^*$-cycle packing of $K_{a,b}$ with a leave $L$ of size $\ell$ with $k$ non-trivial components such that exactly one vertex of $L$ has degree $4$, and every other vertex of $L$ has degree $2$ or degree $0$. If $m_1$ and $m_2$ are even integers such that $m_1,m_2\geq \max(4,k+1)$ and $m_1+m_2=\ell$, then there is an $(M,m_1,m_2)^*$-cycle decomposition of $K_{a,b}$.
\end{theorem}

A vertex $x$ in a connected graph $G$ is said to be a \textit{cut vertex} if the removal of $x$ from $G$ makes the resulting graph disconnected.

\begin{theorem}\label{lemma3.4H}
{\normalfont\cite{Horsley}} Let $a$ and $b$ be positive integers. Suppose that there exists an $(M)^*$-cycle packing of $K_{a,b}$ with a leave $L$. If $u$ and $v$ are vertices in the same part of $K_{a,b}$ such that $\deg_L(u)>\deg_L(v)$, then there exists an $(M)^*$-cycle packing of $K_{a,b}$ with a leave $L'$ such that $\deg_{L'}(u)=\deg_L(u)-2$, $\deg_{L'}(v)=\deg_L(v)+2$, and $\deg_{L'}(x)=\deg_L(x)$ for all $x\in V(L)\setminus \{u,v\}$. Furthermore, this $L'$ also satisfies
\begin{enumerate}
\item[$(i)$] if $\deg_L(v)=0$ and $u$ is not a cut vertex of $L$, then $L'$ has the same number of non-trivial components as $L$; and
\item[$(ii)$] if $\deg_L(v)=0$ then either $L'$ has the same number of non-trivial components as $L$ or $L'$ has one more non-trivial component than $L$.
\end{enumerate}
\end{theorem}

\begin{theorem}\label{lemma3.5H}
{\normalfont\cite{Horsley}} Let $a$ and $b$ be positive integers. For an $(M)^*$-cycle packing $\mathcal{P}$ of a graph $G$ let
$$d(\mathcal{P})=\frac{1}{2}\sum_{x\in D}(\deg_L(x)-2),$$
where $L$ is the leave of $\mathcal{P}$ and $D$ is the set of vertices of $L$ having degree at least $4$. 
Suppose that there exists an $(M)^*$-cycle packing $\mathcal{P}_0$ of $K_{a,b}$, where $a\leq b$, with a leave $L_0$ of size $\ell$, where $\ell\leq 2a+2$ if $a<b$ and $\ell\leq 2a$ if $a=b$, with $k_0$ non-trivial components such that $L_0$ has at least one vertex of degree at least $4$. Then, there exists an $(M)^*$-cycle packing of $K_{a,b}$ with a leave $L'$ such that exactly one vertex of $L'$ has degree $4$, every other vertex of $L'$ has degree $2$ or degree $0$, and $L'$ has at most $\min(k_0+d(\mathcal{P}_0)-1,\lfloor\frac{\ell}{4}\rfloor-1)$ non-trivial components.
\end{theorem}

\begin{theorem}\label{lemma3.6H}
{\normalfont\cite{Horsley}} Let $a$ and $b$ be positive integers. Suppose there exists an $(M,h,m,m')^*$-cycle decomposition of $K_{a,b}$ where $a\leq b$. If $m+m'\leq 3h$, $m+m'+h\leq 2a+2$ if $a<b$, and $m+m'+h\leq 2a$ if $a=b$, then there exists an $(M,h,m+m')^*$-cycle decomposition of $K_{a,b}$. 
\end{theorem}


This next lemma is the multigraph analogue to Theorem~{\normalfont\ref{lemma2.3H}}. Even though one of the assumptions in the proof of Theorem~{\normalfont\ref{lemma2.3H}} requires $v$ and $u$ to be even, this fact is only used to show that the degree of $x_0$ and $x_t$ in $H$ are both $2$ which we can do since $\l v\equiv \l u\equiv 0\pmod{2}$.

\begin{lemma}\label{lemma2.3}
Suppose there exists an $(M)^*$-cycle packing $\mathcal{P}$ of $\l K_{v,u}$ with a leave of size $\ell$ whose only non-trivial component $H$ contains a path $P=[x_0,x_1,\ldots,x_t]$ of even length at least $4$ such that the edges in $E(H)\setminus E(P)$ form a path and such that $x_1x_t\not\in E(H)$. Let $S$ be the $(x_0,x_t)$-switch with origin $x_1$ $($note that $x_0$ and $x_t$ are in the same part of $\l K_{v,u}$ since $P$ is a path of even length$)$ and let $\mathcal{P}'$ be the $(M)^*$-cycle packing of $\l K_{v,u}$ obtained from $\mathcal{P}$ by performing $S$. If $S$ does not have terminus $x_{t-1}$, then the leave of $\mathcal{P}'$ has a decomposition into a $t$-cycle and an $(\ell-t)$-cycle, and there are at least as many vertices of degree $4$ in the leave of $\mathcal{P}'$ as there are in the leave of $\mathcal{P}$.
\end{lemma}

\begin{proof}
By our assumptions, the degree of each vertex in the leave of $\mathcal{P}$ must have even degree and further $\deg_H(x_0)=\deg_H(x_t)=2$. Since Lemma~{\normalfont\ref{cycleSwitching}} is no different than Theorem~{\normalfont\ref{lemma2.1}} if the leave is a simple graph when comparing $\l K_{v,u}$ to $K_{v,u}$, the result follows from Theorem~{\normalfont\ref{lemma2.1}} in the case when there are no $2$-cycles in the leave of $\mathcal{P}$. So suppose there is at least one $2$-cycle in the leave of $\mathcal{P}$. Let the path containing the edges $E(H)\setminus E(P)$ be $[y_0(=x_0),y_1,y_2,\ldots,y_{\ell-t}(=x_t)]$. It is clear that either $y_1$ or $y_{\ell-t-1}$ are the terminus of $S$ since it is supposed that $x_{t-1}$ is not a terminus. If $S$ has terminus $y_1$ then the leave of $\mathcal{P}'$ can be decomposed into a $t$-cycle $(x_1,x_2,\ldots,x_t)$ and an $(\ell-t)$-cycle $(y_1,y_2,\ldots,y_{\ell-t})$ with at least as many vertices of degree $4$ as in the leave of $\mathcal{P}$. Suppose $y_{\ell-t-1}$ is the terminus. Then $y_{\ell-t-1}\neq x_{t-1}$ since $x_{t-1}$ is not a terminus by assumption.
By performing $S$, the leave of $\mathcal{P}'$ can be decomposed into a $t$-cycle $(x_1,x_2,\ldots,x_{t})$ and an $(\ell-t)$-cycle $(y_0,y_1,\ldots,y_{\ell-t-1})$ with at least as many vertices of degree $4$ as in the leave of $\mathcal{P}$. 
\end{proof}

The following lemma is the analogue to Theorems~{\normalfont\ref{lemma2.5H}} and {\normalfont\ref{lemma2.6H}}.

\begin{lemma}\label{lemma2.6}
Suppose there exists an $(M)^*$-cycle packing of $\l K_{v,u}$ with a leave whose only non-trivial component is a $(p,q)$-chain. If $m_1$ and $m_2$ are even integers such that $m_1,m_2\geq 2$ and $m_1+m_2=p+q$, then there exists an $(M,m_1,m_2)^*$-decomposition of $\l K_{v,u}$.
\end{lemma}

\begin{proof}
If $p,q\geq 3$, then we are done by Theorem~{\normalfont\ref{lemma2.6H}}. If $p=q=2$, then $m_1=m_2=2$ and we are done. Let $p=2$, so if $m_1=2$, $p=m_1$ and $q=m_2$. Then $2<m_1\leq m_2< q$, and let the $(2,q)$-chain be $(x_0,c)\cdot(c,y_1,y_2,\ldots,y_{q-1})$. Let $\mathcal{P}'$ be the $(M)^*$-cycle packing of $\l K_{v,u}$ obtained from $\mathcal{P}$ by performing the $(x_0,y_{m_2-1})$-switch with origin $c$ ($x_0$ and $y_{m_2-1}$ are twin since $m_2-1$ is odd). Then regardless of the terminus, we form an $m_2$-cycle and a $(q+2-m_2=m_1)$-cycle in the leave of $\mathcal{P}'$. 
\end{proof}

This next lemma is the multigraph version of Theorem~{\normalfont\ref{lemma3.1H}}.

\begin{lemma}\label{lemma3.1}
Suppose there exists an $(M)^*$-cycle packing of $\l K_{v,u}$, where $v\leq u$, with a leave of size $\ell$, where $\ell\leq 2v+2$ if $v<u$ and $\ell\leq 2v$ if $v=u$, with only one non-trivial component $H$. If $m_1$ and $m_2$ are even integers such that $m_1,m_2\geq 2$ and either 
\begin{itemize}
\item $H$ is a chain which has a decomposition into an $m_1$-path and an $m_2$-path; or
\item $H$ is a ring which has a decomposition into an $m_1$-path and an $m_2$-path;
\end{itemize}
then there exists an $(M,m_1,m_2)^*$-cycle decomposition of $\l K_{v,u}$.
\end{lemma}

\begin{proof}
Let $\mathcal{P}$ be an $(M)^*$-cycle packing of $G=\l K_{v,u}$ with vertex partition $\{V,U\}$ that satisfies the conditions of this lemma and let $L$ be its leave. Let $P=[x_0,x_1,\ldots,x_{m_1-1},x_{m_1}]$ be an $m_1$-path in $H$. The result will be shown using induction on the number of cycles in the ring or chain, $s$. Note that if $H$ contains no $2$-cycles, we are done by Theorem~{\normalfont\ref{lemma3.1H}} so let $H$ contain a $2$-cycle.

If $s=2$ and $H$ is a chain then we are done by Lemma~{\normalfont\ref{lemma2.6}}. Suppose $s=2$ and $H$ is a ring. Then there are two vertices of degree $4$ in $H$. Since $s=2$ and $H$ has a $2$-cycle, the two vertices of degree $4$ are adjacent, and hence in different parts. Let $\alpha$ be the vertex of degree $4$ in $U$. There exists a vertex $\beta\in U$ (that is hence twin to $\alpha$ in $G$) such that $\deg_H(\beta)=0$ for the following reason. Note that every vertex in $H$ has even degree.  Further, $H$ is bipartite (since $G$ was) and hence $\sum_{u\in U}\deg_H(u)=|L|$.  If no such $\beta$ existed, then $\sum_{u\in U}\deg_H(u)\geq 2u+2$ (where the extra $+2$ comes from the vertex of degree $4$).  But then $2u+2\leq |L|$ a contradiction with the size of $L$. 

Then let $\mathcal{P}'$ be the $(M)^*$-cycle packing of $\l K_{v,u}$ obtained by performing an $(\alpha,\beta)$-switch where the origin is a vertex that is a neighbor of $\alpha$ in $H$. Regardless of the terminus, the resulting leave will be a $2$-chain and so by Lemma~{\normalfont\ref{lemma2.6}}, the intended result follows. Now suppose by way of induction that for all $3\leq t<s$, a $t$-chain or $t$-ring, which has a decomposition into an $m_1$-path and an $m_2$-path, can be decomposed as the lemma states.
Notice that since $s\geq 3$, $m_1,m_2\geq 4$. We shall proceed in two cases depending on whether $H$ is a chain or ring.

{\bf Case 1:} Suppose $H$ is an $s$-chain. It is clear that the degree of $x_0$ and $x_{m_1}$ is $2$ in $L$ since otherwise either $P$ is not a path or $E(L)\setminus E(P)$ is not a path. So $x_0$ and $x_{m_1}$ are in the end cycles of the $s$-chain $H$. Since $m_1$ is even, $x_0$ and $x_{m_1}$ are twin. Let $\mathcal{P}^\dagger$ be the $(M)^*$-cycle packing of $\l K_{v,u}$ obtained by performing the $(x_0,x_{m_1})$-switch $S$ with origin $x_1$ where the leave is $L^\dagger$. Recall that $x_1x_{m_1}\not\in E(H)$ since $H$ is an $s$-chain. So if the terminus of $S$ is not $x_{m_1-1}$ then we are done by Lemma~{\normalfont\ref{lemma2.3}}. Suppose that the terminus of $S$ is $x_{m_1-1}$. Then the resulting leave is an $(s-1)$-ring and $P^\dagger=(E(P)\setminus\{x_0x_1,x_{m_1-1}x_{m_1}\})\cup\{x_0x_{m_1-1},x_1x_{m_1}\}$ is an $m_1$-path in $L^\dagger$ such that $E(L^\dagger)\setminus E(P^\dagger)$ is an $m_2$-path. Thus, we are done by the induction hypothesis.

{\bf Case 2:} Suppose $H$ is an $s$-ring. If all ring cycles in $H$ have length $2$ then let $\alpha$ be a link vertex in $H$ and $\beta$ be a vertex twin with $\alpha$ in $L$ such that $\deg_L(\beta)=0$ (vertices $\alpha$ and $\beta$ exist for the same reason as in Case 1). 
Then let $\mathcal{P}'$ be the $(M)^*$-cycle packing of $\l K_{v,u}$ obtained by performing the $(\alpha,\beta)$-switch where the origin is a neighbor of $\alpha$ in $H$. The resulting leave will either be an $(s-1)$-ring in which case we are done by induction or an $s$-chain in which case we are back in Case $1$, so the required decomposition is reached. 

So $H$ must contain a ring cycle $C$ that is larger than a $2$-cycle such that $|V(C)\cap V(D)|=1$ where $D$ is a $2$-cycle (recall that there must be at least one $2$-cycle). Let $c$ be the link vertex in $C$ and $D$, let $D=(\alpha,c)$, and let $\beta\in V(L)$ such that $\deg_L(\beta)=0$ (vertices $\alpha$ and $\beta$ exist for the same reason as in Case 1). Then we obtain an $(M)^*$-cycle packing of $\l K_{v,u}$ by performing a $(c,\beta)$-switch with origin $\alpha$ with a leave that is either an $s$-chain in which case we proceed as in Case $1$ or an $(s-1)$-ring in which case we are done by the induction hypothesis.

\end{proof}

The following lemma will use Theorem~{\normalfont\ref{lemma3.2H}} to prove the multigraph case when $2$-cycles are not in the leave. Note that we do not need the $m_1$-path and $m_2$-path to have end vertices in a designated part of $\l K_{v,u}$ as was needed in {\normalfont\cite{Horsley}}.

\begin{lemma}\label{lemma3.2}
Suppose that there exists an $(M)^*$-cycle packing $\mathcal{P}$ of $\l K_{v,u}$ with a leave $L$ of size $\ell$ with $k$ non-trivial components such that exactly one vertex of $L$ has degree $4$ and every other vertex of $L$ has degree $2$ or degree $0$. If 
$m_1$ and $m_2$ are integers such that $m_1,m_2\geq k+1$ and $m_1+m_2=\ell$, then there exists an $(M)^*$-cycle packing of $\l K_{v,u}$ with a leave whose only non-trivial component is a chain which has a decomposition into an $m_1$-path and an $m_2$-path.
\end{lemma}

\begin{proof}
Note that the leave must have one component that is a $2$-chain and all other non-trivial component must be cycles. Let $k_1$ be the number of components of the leave that are $2$-cycles and let $k_2=k-k_1-1$ be the number of components of the leave that are cycles of length at least $3$. 

Recall that Theorem~{\normalfont\ref{lemma3.2H}} holds when the leave contains only cycles that are not $2$-cycles and the $2$-chain does not contain a $2$-cycle. Let $H$ be the $2$-chain and $C_1,C_2,\ldots,C_{k-1}$ be the cycles such that $H\cup C_1\cup C_2\cup \cdots\cup C_{k-1}$ forms the leave of $\mathcal{P}$ and $C_1,C_2,\ldots,C_{k_1}$ are $2$-cycles (if $k=1$ then the leave has a single component $H$). We break this proof into three cases depending on the number of $2$-cycles in $H$.

{\bf Case 1:} Suppose $H$ does not contain a $2$-cycle. 
The breakdown of this case will be to first use Theorem~\ref{lemma3.2H} to join $H$ to all the cycles of length at least $3$ (since there are no $2$-cycles in $H\cup C_{k_1+1}\cup C_{k_1+2}\cup \cdots \cup C_{k-1}$) and then we will join the resulting chain to the $2$-cycles afterward. Notice that by Lemma~{\normalfont\ref{cycleSwitching}}, if $\alpha$ and $\beta$ are two twin vertices in $L$ and some cycle $C_i$, $i\in \mathbb{Z}_k\setminus\{0\}$, does not contain vertices $\alpha$ or $\beta$, then none of the edges in $C_i$ will change after performing an $(\alpha,\beta)$-switch. 
Then by Theorem~{\normalfont\ref{lemma3.2H}}, the $k_2$ cycles of length at least $3$ can be joined to $H$ without changing the length of the $k_1$ $2$-cycles, so that there exists an $(M)^*$-cycle packing $\mathcal{P}'$ of $\l K_{v,u}$ with a leave whose non-trivial components are $k_1$ $2$-cycles and a chain which has a decomposition into an $(m_1-k_1)$-path and an $(m_2-k_1)$-path. 
We will form a sequence of $(M)^*$-cycle packings of $\l K_{v,u}$ $\mathcal{P}_1,\mathcal{P}_2,\ldots,\mathcal{P}_{k_1}$ as follows.
Let $P_1=[x_0,x_1,\ldots,x_{m_1-k_1}]$ and $P'_1=[y_0,y_1,\ldots,y_{m_2-k_1}]$ be the paths that decompose the chain in the leave of $\mathcal{P}_1=\mathcal{P}'$. Since $x_0=y_0$, let $C_1=(z_0,z_1)$ such that $z_0$ and $x_0$ are twin (since our graph is bipartite). Then let $\mathcal{P}_2$ be the $(M)^*$-cycle packing of $\l K_{v,u}$ obtained by performing the $(z_0,x_0)$-switch $S$ with origin $x_1$ and let $L_2$ be the leave of $\mathcal{P}_2$. Regardless of the terminus of $S$, the resulting leave $L_2$ will have $k_1-1$ disjoint $2$-cycles and the chain will have two more edges so that each path $P_1$ and $P'_1$ can be extended by a single edge to form new paths $P_2$ and $P_2'$. By repeating this process $k_1$ times, we extend each path to form $P_{k_1}$ and $P_{k_1}'$ that together decompose the leave of $\mathcal{P}_{k_1}$ so that we end up with a single chain that can be decomposed into an $m_1$-path and an $m_2$-path.

{\bf Case 2:} Suppose $H$ contains exactly one $2$-cycle. 
We will form a sequence of $(M)^*$-cycle packings of $\l K_{v,u}$ 
$\mathcal{P}_1,\mathcal{P}_2,\ldots,\mathcal{P}_{k_2-1}$ as follows.
Let $x_1\in V(H)$ such that $x_1$ is not in the $2$-cycle in $H$, let $y_1$ be a neighbor of $x_1$ in $H$. If there exists a cycle of length at least $3$, let $z_1\in V(C_{k_1+1})$ (recall that $C_{k_1+1}$ is \textit{not} a $2$-cycle) such that $z_1$ and $x_1$ are twin (we can guarantee such twin vertices exist since our graph is bipartite). If there is no such cycle, that is, if $k_2=0$. Then the leave is a $2$-chain with a $2$-cycle and an $(m_1+m_2-2k_1-2)$-cycle. This $2$-chain can easily be decomposed into an $(m_1-k_1)$-path and an $(m_2-k_1)$-path. We proceed to join this $2$-chain with the remaining $2$-cycles as we did in Case $1$ to form a single chain which has a decomposition into an $m_1$-path and an $m_2$-path.
Otherwise, if $k_2\neq 0$, we proceed as follows. 
Let $\mathcal{P}_1$ be the $(M)^*$-cycle packing of $\l K_{v,u}$ obtained by performing the $(x_1,z_1)$-switch with origin $y_1$ and let $H_2$ be the component of the leave of $\mathcal{P}_1$ that is a chain ($H_2$ is either a $2$-chain or $3$-chain). 
If $H_2$ is a $3$-chain where the $2$-cycle is an end cycle, then there is an isolated vertex in $L$ (since $m_1+m_2<v+u$ and $|V(H_2)|<m_1+m_2-2$), say $b$, in the same part of the partition $V(\l K_{v,u})$ as the link vertex, $x'$, of the $2$-cycle in $H_2$, and so let $\mathcal{P}_1'$ be the $(M)^*$-cycle packing of $\l K_{v,u}$ obtained by performing the $(b,x')$-switch $S$ where the origin of $S$ is a neighbor of $x'$ and let $H_2'$ be the chain in the leave of $\mathcal{P}_1'$. Regardless of the terminus, $H_2'$ is a $2$-chain with no $2$-cycles in which case we proceed as in Case $1$. So we may assume that $H_2$ is a $2$-chain that contains exactly one $2$-cycle. 
We proceed in the following inductive manner to form a sequence $\mathcal{P}_1,\mathcal{P}_2,\ldots,\mathcal{P}_{k_2-1}$ of cycle packings of $\l K_{v,u}$ so that the leave contains fewer and fewer cycles with lengths at least $3$. 
As in the beginning of this case, let $x_i\in V(H_i)$ such that $x_i$ is not in the $2$-cycle, let $y_i$ be a neighbor of $x_i$ in $H_i$, and let $z_i\in V(C_{k_1+i})$ such that $x_i$ and $z_i$ are twin (since our graph is bipartite), so then let $\mathcal{P}_i$ be the $(M)^*$-cycle packing of $\l K_{v,u}$ obtained by performing the $(x_i,z_i)$-switch with origin $y_i$ and let $H_{i+1}$ be the component of the leave of $\mathcal{P}_i$ that is a chain for $i\in\{1,2,\ldots,k_2-1\}$. If $H_{i+1}$ is ever a $3$-chain, we can perform the same cycle switch as described earlier and proceed as in Case $1$. So $H_{k_2}$ is a $2$-chain with a $2$-cycle and an $(m_1+m_2-2k_1-2)$-cycle. This $2$-chain can easily be decomposed into an $(m_1-k_1)$-path and an $(m_2-k_1)$-path. Now we proceed to join this $2$-chain with the remaining $2$-cycles as we did in Case $1$ to form a single chain which has a decomposition into an $m_1$-path and an $m_2$-path.

{\bf Case 3:} Suppose $H$ contains exactly two $2$-cycles. Suppose that $k_2=0$. Then there are $k_1+1=k$ components, $m_1+m_2=\ell=2k_1+4=2k-2+4=2k+2$ and $m_1=m_2=k+1$. Let $a$ be a vertex of degree $2$ in $H$ and $b$ be a vertex in $C_1$ that is twin with $a$ (which exists since our graph is bipartite). Then let $\mathcal{P}^\dagger$ be the $(M)^*$-cycle packing of $\l K_{v,u}$ obtained by performing the $(a,b)$-switch where the origin is in $C_1$, so the resulting leave will have one less component and the leave will either be a $3$-chain of $2$-cycles (if the terminus is in $C_1$) or a $2$-chain with a $4$-cycle and a $2$-cycle (if the terminus is not in $C_1$) which in either case can be decomposed into two $3$-paths. In the latter case we are in Case $2$.
Thus we continue this process for the remaining $k_1-1$ $2$-cycles to form a chain that can be decomposed into an $m_1$-path and an $m_2$-path. 

Now suppose $k_2>0$. Let $H=(x,c)\cdot(c,y)$ and let $C_{k_1+1}=(z_0,z_1,\ldots,z_{t})$ such that $z_0$ is twin with $x$ (which exists since our graph is bipartite). Then let $\mathcal{P}^\ddagger$ be the $(M)^*$-cycle packing of $\l K_{v,u}$ obtained by performing the $(x,z_0)$-switch with origin $z_1$ and let the leave be $L^\ddagger$. Then the leave is a collection of vertex-disjoint $2$-cycles and either a $2$-chain with one cycle that is not a $2$-cycle or a $3$-chain with two $2$-cycles. In the former case we proceed as in Case $2$. We continue assuming the latter case, so let $y'$ be the link vertex between two of the $2$-cycles. Since there is now at least one isolated vertex in $L^\ddagger$, say $b$, that is twin with $y'$ (since our graph is bipartite), 
let $\mathcal{P}''$ be the $(M)^*$-cycle packing of $\l K_{v,u}$ obtained by performing the $(y',b)$-switch where the origin is a neighbor of $y'$ in the $3$-chain and let the leave of $\mathcal{P}''$ be $L''$. Then the resulting chain in $L''$ will either be a $2$-chain with no $2$-cycles or a $2$-chain with exactly one $2$-cycle. In the former case, we proceed as in Case $1$ and in the latter case, we proceed as in Case $2$.
\end{proof}

For completeness, we state below the analogous lemmas to Theorems~{\normalfont\ref{lemma3.3H}}, {\normalfont\ref{lemma3.4H}}, and {\normalfont\ref{lemma3.5H}} respectively. Their proofs follow mostly from the corresponding proofs in {\normalfont\cite{Horsley}}. Note that Lemma~{\normalfont\ref{lemma3.5}} is different to Theorem~{\normalfont\ref{lemma3.5H}} in that the latter result has at most $\min(\{k_0+d(\mathcal{P}_0)-1, \lfloor \frac{\ell}{4}\rfloor-1\})$ non-trivial components in the leave $L'$. This is a direct consequence of $2$-cycles existing in multigraphs and not in simple graphs. Again, it is clear that the same proof with just the obvious generalizations works for multigraphs. It should be noted that even though $v$ and $u$ are even for Theorems~{\normalfont\ref{lemma3.3H}}, {\normalfont\ref{lemma3.4H}}, and {\normalfont\ref{lemma3.5H}}, that fact is only used to show that each vertex in the leave has even degree. If there exists an $(M)^*$-cycle packing of $\l K_{v,u}$, then each vertex in the leave has even degree.

\begin{lemma}\label{lemma3.3}
Suppose that there exists an $(M)^*$-cycle packing of $\l K_{v,u}$ with a leave $L$ of size $\ell$ with $k$ non-trivial components such that exactly one vertex of $L$ has degree $4$, and every other vertex of $L$ has degree $2$ or degree $0$. If $m_1$ and $m_2$ are even integers such that $m_1,m_2\geq \max(\{2,k+1\})$ and $m_1+m_2=\ell$, then there is an $(M,m_1,m_2)^*$-cycle decomposition of $\l K_{v,u}$.
\end{lemma}

\begin{proof}
This result is a direct consequence of applying Lemma~{\normalfont\ref{lemma3.2}} to $\mathcal{P}$ to obtain an $(M)^*$-cycle packing of $\l K_{v,u}$ with the appropriate leave and then applying Lemma~{\normalfont\ref{lemma3.1}} to attain the required cycle decomposition.
\end{proof}

\begin{lemma}\label{lemma3.4}
Suppose that there exists an $(M)^*$-cycle packing of $\l K_{v,u}$ with a leave $L$. If $a$ and $b$ are vertices in the same part of $\l K_{v,u}$ such that $\deg_L(a)>\deg_L(b)$, then there exists an $(M)^*$-cycle packing of $\l K_{v,u}$ with a leave $L'$ such that $\deg_{L'}(a)=\deg_L(a)-2$, $\deg_{L'}(b)=\deg_L(b)+2$, and $\deg_{L'}(x)=\deg_L(x)$ for all $x\in V(L)\setminus\{a,b\}$. Furthermore, this $L'$ also satisfies
\begin{itemize}
\item[$(i)$] if $\deg_L(b)=0$ and $a$ is not a cut vertex of $L$, then $L'$ has the same number of non-trivial components as $L$; and
\item[$(ii)$] if $\deg_L(b)=0$, then either $L'$ has the same number of non-trivial components as $L$ or $L'$ has one more non-trivial component than $L$.
\end{itemize}
\end{lemma}

\begin{proof}
Suppose $\mathcal{P}$ is an $(M)^*$-cycle packing of $\l K_{v,u}$ as described in this lemma. Since $L$ is even, it follows that $\deg_L(a)\geq \deg_L(b)+2$. As a consequence of this and the fact that $a$ and $b$ are twins, there exists an $(a,b)$-switch whose origin and terminus are neighbors in $L$ of $a$. It then follows that the properties in the lemma are satisfied as a result of $L$ being an even graph.
\end{proof}

\begin{lemma}\label{lemma3.5}
Suppose that there exists an $(M)^*$-cycle packing $\mathcal{P}_0$ of $\l K_{v,u}$, where $v\leq u$, with a leave $L_0$ of size $\ell$, where $\ell\leq 2v+2$ if $v< u$ and $\ell \leq 2v$ if $v=u$, with $k_0$ non-trivial components such that $L_0$ has at least one vertex of degree at least $4$. Then, there exists an $(M)^*$-cycle packing of $\l K_{v,u}$ with a leave $L'$ such that exactly one vertex of $L'$ has degree $4$, every other vertex of $L'$ has degree $2$ or degree $0$, and $L'$ has at most $\min(\{k_0+d(\mathcal{P}_0)-1, \lfloor \frac{\ell}{2}\rfloor-1\})$ non-trivial components where
$$d(\mathcal{P}_0)=\frac{1}{2}\sum_{x\in D}(\deg_L(x)-2),$$
where $D$ is the set of vertices of $L$ having degree at least $4$.
\end{lemma}

\begin{proof}
Let $d=d(\mathcal{P}_0)$. At minimum, Lemma~{\normalfont\ref{lemma3.4}} must be applied $d-1$ times to form the required packing from $\mathcal{P}_0$ into an $(M)^*$-cycle packing with a leave containing exactly one vertex of degree $4$ and all other vertices of degree $2$ or $0$.

Let $\mathcal{P}_0,\mathcal{P}_1,\ldots,\mathcal{P}_{d-1}$ be a sequence of $(M)^*$-cycle packing of $\l K_{v,u}$ formed in the following inductive manner. The $(M)^*$-cycle packing $\mathcal{P}_{i+1}$ with leave $L_{i+1}$ and $k_{i+1}$ components is formed from $\mathcal{P}_i$ by choosing one isolated vertex $a$ in the leave of $\mathcal{P}_i$ and a twin vertex $b$ (by choosing vertices in the same part as described in the proof to Lemma~{\normalfont\ref{lemma3.1}}) of degree $4$ or more in the leave of $\mathcal{P}_i$. Such vertices $a$ and $b$ exist (depending on whether or not all the vertices of degree at least $4$ are in the same part) based on the number of edges in $L$, and based on the fact that any $(M)^*$-cycle packing in the sequence given above save for $\mathcal{P}_{d-1}$ has either a single vertex of degree at least $6$ or at least $2$ vertices of degree $4$. 

Now it remains to show that $k_{d-1}\leq \min(\{k_0+d-1,\left\lfloor \frac{\ell}{2}\right\rfloor-1)$.
It is clear that $L_{d-1}$ has a decomposition into $k_{d-1}+1$ cycles since the largest degree occurs at a single vertex in $L_{d-1}$ with degree $4$. Each of these cycles have length at least $2$ which shows $\ell\geq 2(k_{d-1}+1)$, thus $k_{d-1}\leq \left\lfloor\frac{\ell}{2}\right\rfloor-1$. By Property $(ii)$ of Lemma~{\normalfont\ref{lemma3.4}}, $k_{i+1}\leq k_i+1$ for each $i\in\{0,1,\ldots,d-2\}$ and so $k_{d-1}\leq k_0+d-1$, attaining the desired conclusion.
\end{proof}

\begin{lemma}\label{multiBipartitePacking}
Suppose there exists an $(M,h,m,m')^*$-cycle decomposition of $\l K_{v,u}$, where $v\leq u$. If 
\begin{itemize}
\item $m+m'\leq h$, 
\item $m+m'+h\leq 2v+2$ if $v<u$, and 
\item $m+m'+h\leq 2v$ if $v=u$, 
\end{itemize}
then there exists an $(M,h,m+m')^*$-cycle decomposition of $\l K_{v,u}$.
\end{lemma}

\begin{proof}
This lemma was proved for simple graphs in {\normalfont\cite{Horsley}}, but the same proof can be generalized for multigraphs as long as the leave does not contain a $2$-cycle (with the additional change that $m+m'\leq h$). Further, we can use Lemmas~{\normalfont\ref{lemma3.3}}, {\normalfont\ref{lemma3.4}}, and {\normalfont\ref{lemma3.5}} as was done in {\normalfont\cite{Horsley}} to prove this lemma with just the obvious generalizations for the first condition in this lemma. This modification  is a direct result of $2$-cycles being a possible cycle in $\l K_{v,u}$. To satisfy the conditions in Lemma~\ref{lemma3.3}, $m+m'$ and $h$ must be greater than $k$, the number of non-trivial components in the leave but when we use Lemma~\ref{lemma3.5}, $k\leq \min(\{m+m'-1,\lceil\frac{m+m'+h}{2}\rceil-1\})$, hence our first condition. With this explanation above for the modified first condition of the lemma, follow the proof in Theorem 3.6 of \cite{Horsley} to attain the desired result.
\end{proof}

The lemma above should come as no surprise since Lemma~{\normalfont{\ref{multiBipartitePacking}} is a multigraph version of Theorem~{\normalfont\ref{lemma3.6H}} (except it turns out that when $2$-cycles are involved, we require $m+m'\leq h$ instead of $m+m'\leq 3h$).

\section{Main Theorem}\label{mainTheorem}

To attain our main result, we first need this next result from {\normalfont\cite{Horsley}}. 

\begin{theorem}\label{bipartitePacking}
{\normalfont\cite{Horsley}} Let $a$ and $b$ be positive integers such that either $a$ and $b$ are even or $a=b$, and let $K^*_{a,b}$ be the graph $K_{a,b}$ if $a$ and $b$ are even and the graph $K_{a,b}-I$ if $a=b$ and $a$ is odd where $I$ is a $1$-factor. If $m_1,m_2,\ldots,m_t$ are even integers such that $4\leq m_1\leq m_2\leq\cdots \leq m_t\leq \min(a,b,3m_{t-1})$ and $m_1+m_2+\cdots+m_t=|E(K_{a,b}^*)|$, then there is a cycle decomposition of $K_{a,b}^*$ into cycles of lengths $m_1,m_2,\ldots,m_t$.
\end{theorem}


Let $M_1=m_1',m_2',\ldots,m_s'$ and $M_2=m_1'',m_2'',\ldots,m_t''$ be non-decreasing sequences of positive integers.
We say $M_1<M_2$ if and only if $s>t$ or $s=t$ and $m_k'<m_k''$ where $k$ is the smallest positive integer such that $m_k'\neq m_k''$. Thus we have an ordering to the sequences representing the cycle sizes. Based on the definition, the next lemma proves our main result for the smallest cases. 
Recall that $\nu_k(M)$ is the number of times $k$ appears in the sequence $M$. 

\begin{lemma}\label{smallBipartiteDecomp}
Let $\l$, $u$, and $v$ be positive integers satisfying the necessary conditions found in Theorem~{\normalfont\ref{Nconditions}}. If $\l$ is even, let $M_1=2,2,\ldots,2,4,4,\ldots,4,m_t$ be a non-decreasing sequence of $t$ terms such that $\nu_4(M)=k$ and $\nu_2(M)=t-k-1$. If $\l$ is odd, let $M_2=2,2,\ldots,2,4,4,\ldots,4,m_{t-1},m_t$ be a non-decreasing sequence of $t$ terms such that $\nu_4(M)=k$ and $\nu_2(M)=t-k-2$, and $m_t\leq \min(\{v,u,3m_{t-1}\})$ satisfies the conditions in Theorem~{\normalfont\ref{Nconditions}}. Regardless of the value of $\lambda$, let $k$ be as small as possible and still satisfy the necessary conditions in Theorem~{\normalfont\ref{Nconditions}}. Then there exists an $(M_1)$-cycle decomposition $\mathcal{P}$ of $\l K_{v,u}$ when $\l$ is even and there exists an $(M_2)$-cycle decomposition $\mathcal{P}$ of $\l K_{v,u}$ when $\l$ is odd.

\end{lemma}

\begin{proof}
Let $G=\l K_{v,u}$. Suppose $\l$ is even. Then $\l vu=m_t+4k+2(t-k-1)$ implies $m_t=\l vu-2k-2t+2$. So 
$$k+\frac{m_t-2}{2}+t=\frac{\l}{2} vu \geq t+m_t-2.$$
The last inequality is by Theorem~\ref{anyGraphNecessary}$(c)$. It follows that
$$k\geq m_t-2-\frac{m_t-2}{2}=\frac{2m_t-4-m_t+2}{2}=\frac{m_t-2}{2}.$$
Since we want to minimize $k$, there are $k=\frac{m_t-2}{2}$ $4$-cycles in $\mathcal{P}$. Let $C_{m_t}=(x_0,x_1,\ldots,x_{m_t-1})$ be a cycle of length $m_t$. Also, let $C_{\frac{m_t+2}{2}+k}=(x_0,x_{2k-1},x_{2k},x_{2k+1})$ for $k\in\{1,2,\ldots,\frac{m_t-2}{2}\}$ be cycles of length $4$. Then the resulting subgraph $G'=C_{\frac{m_t+2}{2}}\cup C_{\frac{m_t+4}{2}}\cup\cdots\cup C_{m_t}$ is an even subgraph of $G$ where for each pair of vertices $x$ and $y$ in $G$, $\mu_{G'}(xy)\in \{0,2\}$. It is clear that the remainder of the graph can be decomposed into $2$-cycles. Thus, the required cycle decomposition exists.

Suppose $\l$ is odd. By Theorem~\ref{anyGraphNecessary}$(d)$ we need $k=\lceil \frac{vu-m_t-m_{t-1}}{4}\rceil$ $4$-cycles. If $4$ divides $vu-m_t-m_{t-1}$ then by Theorem~{\normalfont\ref{bipartitePacking}} there exists a $(4,4,\ldots,4,m_{t-1},m_t)$-cycle decomposition of $K_{v,u}$ and it is clear that there is a decomposition of $(\l-1)K_{v,u}$ into $2$-cycles. Suppose $4$ does not divide $vu-m_t-m_{t-1}$. Since $\l v\equiv \l u\equiv 0\pmod{2}$ and $\l$ is odd, $v\equiv u\equiv 0\pmod{2}$, and so there exists a $(4,4,\ldots,4,m_{t-1},m_t)$-cycle packing $\mathcal{P}$ of $K_{v,u}$ with a leave that is a $6$-cycle, $C_0=(x_0,x_1,\ldots,x_5)$ by Theorem~{\normalfont\ref{bipartitePacking}}. Then $C_1=(x_0,x_1,x_2,x_3)$ and $C_2=(x_0,x_3,x_4,x_5)$ are the remaining $4$-cycles and thus $\mathcal{P}\cup \{E(C_1)\cup E(C_2)\}$ is an $(M_2)$-cycle decomposition of $K_{v,u}$ with an additional $2$-cycle between $x_0$ and $x_3$. So the remaining edges of $\l K_{v,u}$ can be decomposed into $2$-cycles.
\end{proof}

\begin{theorem}\label{bipartiteMaxPacking}
Let $v$, $u$, and $\l$ be positive integers such that $v,u\geq 5$, $v\leq u$, and $\l v\equiv \l u\equiv 0\pmod{2}$. Let $M=m_1,m_2,\ldots,m_t$ be a non-decreasing sequence of positive even integers such that $2\leq m_1$. An $(M)$-cycle decomposition of $\l K_{v,u}$ exists if all of the following hold:
\begin{itemize}
\item $m_t\leq 3m_{t-1}$, 
\item $t\leq \frac{\l}{2}vu-m_t+2$ if $\l$ is even,
\item $2\nu_2(M)\leq (\l-1)vu$,
\item $m_{t-1}+m_t\leq 2v+2$ if $v<u$,
\item $m_{t-1}+m_t\leq 2v$ if $v=u$, and 
\item $m_1+ m_2+\cdots+ m_t= \l vu$.
\end{itemize}
\end{theorem}

\begin{proof}
Let $v\leq u$. If $\l$ is even, by Lemma~{\normalfont\ref{smallBipartiteDecomp}}, for any $m_t\in\{4,6,\ldots,2v-2\}$ there exists a $(2,2,\ldots,2,4,4,\ldots,4,m_t)$-cycle decomposition $\mathcal{P}$ of $\l K_{v,u}$ where the number of $4$-cycles is as small as possible to satisfy the necessary conditions in Lemma~{\normalfont\ref{Nconditions}}. 
If $\l$ is odd, by Lemma~{\normalfont\ref{smallBipartiteDecomp}} there exists a $(2,2,\ldots,2,4,4,\ldots,4,m_{t-1},m_t)$-cycle decomposition $\mathcal{P}$ of $\l K_{v,u}$ where the sum of the edges in the $4$-cycles and $m_{t-1}$ and $m_t$ is $vu$ or $vu+2$.
Since $m_{i}\leq m_t$ for all $i\in\{1,2,\ldots,t-1\}$  when $\l$ is even or $i\in \{1,2,\ldots,t-2\}$ when $\l$ is odd, and by Lemma~{\normalfont\ref{multiBipartitePacking}} (with $m=4$ or $m=2$, $m'=n_i-4$ or $m'=n_i-2$ where $n_i\leq m_i$ is an even integer, and $h=m_t$), it follows that cycles can be joined in $\mathcal{P}$ to form an $(m_1,m_2,\ldots,m_t)$-cycle decomposition of $\l K_{v,u}$. 

\end{proof}


\bibliographystyle{plain}
\bibliography{vdec}
\end{document}